\documentclass[11pt,letterpaper]{amsart}
\usepackage{amsmath}
\usepackage{amssymb}

\usepackage{epsfig}
\usepackage[all, knot]{xy}
\xyoption{arc}

\newtheorem{thm}{Theorem}[section]
\newtheorem{df}[thm]{Definition}
\newtheorem{prop}[thm]{Proposition}

\newtheorem{lem}[thm]{Lemma}

\newcommand{\Div}{\operatorname{Div}}

\newcommand{\Pic}{\operatorname{Pic}}

\newcommand{\ord}{\operatorname{ord}}
\newcommand{\pp}{\mathbb{P}}
\newcommand{\af}{\mathbb{A}}
\newcommand{\zz}{\mathbb{Z}}

\begin{document}

\title[A height bound for regular affine automorphisms]
{An upper bound for the height for regular affine automorphisms
of $\af^n$}

\author{Chong Gyu Lee}


\date{\today}


\address{Department of Mathematics, Brown University, Providence RI 02912, US}

\email{phiel@math.brown.edu}

\maketitle

\begin{abstract}
In \cite{K}, Kawaguchi proved a lower bound for height of
$h\bigl(f(P)\bigr)$ when~$f$ is a regular affine automorphism of
$\af^2$, and he conjectured that a similar estimate is also true for
regular affine automorphisms of $\af^n$ for $n\geq 3$. In this paper
we prove Kawaguchi's conjecture. This implies that Kawaguchi's theory
of canonical heights for regular affine automorphisms of projective
space is true in all dimensions.
\end{abstract}

\section{Introduction}

Let $\zeta$ be a rational map on $\pp^n$ and $Z(\zeta)$ be the
indeterminacy locus of $\zeta$.  We will say that a family of rational
maps $\{ \zeta_i \}$ is jointly regular if $\bigcap Z(\zeta_i)$ is
empty.  Silverman~\cite{S1} proved the following result for jointly
regular maps.

\begin{thm}
If $\zeta_1 , \cdots , \zeta_m : \pp^n \rightarrow
\pp^n$  is a family of  jointly regular maps, then
\[
  \sum_{i=1}^m \dfrac{1}{d_i} h(\zeta_i(P)) \ge h(P)  - C,
\]
where $d_i$ be the degree of $\zeta_i$, and the constant~$C$ is
independent of the point~$P$.
\end{thm}

In special cases, we can improve the bound. Suppose that
\[f : \af^n \rightarrow \af^n\]
is an affine automorphism with inverse function $f^{-1}$.  Since
$\af^n$ is a dense subset of $\pp^n$, we can find $\phi_0$ and
$\psi_0$, rational functions that extend $f$ and $f^{-1}$,
\[
 \xymatrix{
     \pp^n    \ar@{-->}[d]
       & \pp^n  \ar@{-->}[l]_{\psi_0} \ar@{-->}[r]^{\phi_0}  \ar@{-->}[d]
       &  \pp^n  \ar@{-->}[d] \\
     \af^n
     & \af^n \ar[l]_{f^{-1}}    \ar[r]^{f}
     & \af^n
     }
\]
We say $f$ is a regular affine automorphism if $\{\phi_0,\psi_0 \}$ is
jointly regular. Kawaguchi showed in \cite{K} that Silverman's result
can be improved when we have a regular affine automorphism of~$\af^2$.

\begin{thm}
\label{thm:kawa1}
Let $f$ be a regular affine automorphism on $\af^2$. Then,
\[
   \dfrac{1}{\deg f} h(f(P)) + \dfrac{1}{\deg f^{-1}}h(f^{-1}(P))
   > \left( 1+ \dfrac{1}{\deg f \deg f^{-1}} \right) h(P) + C
\]
\end{thm}

Kawaguchi conjectured that Theorem~\ref{thm:kawa1} is true for regular
affine automorphisms of $\af^n$ for all $n\geq 2$.  In this paper we
prove Kawaguchi's conjecture. From now on, we will let $H$ be the
hyperplane at infinity, $f$ and $f^{-1}$ an affine automorphism and
its inverse, and $\phi_0$ and $\psi_0$ morphisms that are meromorphic
extensions of $f$ and $f^{-1}$. By $Z(\zeta)$ we will mean the
indeterminacy locus of the rational map $\zeta$.

\par\noindent\emph{Acknowledgements}.\enspace
The author would like to thank Joseph H. Silverman for his advice
and Dan Abramovich for his assistance, in particular with the proof of Lemma~\ref{pull-push}.

\section{Resolving Indeterminacy and the Essential Divisor of $\phi$ }

As a standard application of blowing up subschemes, we know that for
any rational map $f : \pp^n \rightarrow \pp^n$, there is a blowup $V$
of $\pp^n$ and a birational morphism $\pi: V \rightarrow \pp^n$ such
that $f \circ \pi$ is a morphism. In general, not all birational
morphism are decomposed into monoidal transformations (blowing up
along closed subvarieties). Fortunately, we have Hironaka's theorem on
resolution of indeterminacy.

\begin{thm}
\label{indeterminacy}
Let $\zeta:V \rightarrow W$ be a rational map between smooth
proper varieties. Then there is a sequence of varieties
$V_i$ such that:
\begin{enumerate}
   \item $V_i$ is a blowup of $V_{i-1}$ along a smooth
       irreducible subvariety.
   \item The meromorphic extension $\zeta_r$ of $f$ on $V_r$ is a
       morphism.
\end{enumerate}
\[
   \xymatrix{
   V_r \ar[d] \ar[rdd]^{\zeta_r} \ar@/_2pc/[dd]_{\pi} \\
   V_i \ar[d] \ar@{-->}[rd]|{\zeta_i} \\
   V=V_0 \ar@{-->}[r]_{\zeta=\zeta_0}  &W }
\]
\end{thm}
\begin{proof}
    \cite[Main Theorem II, Corollary 3]{Hi}
\end{proof}

\begin{df}
We call a birational map $\pi:X \rightarrow Y$ a \emph{monoidal
transformation} if $X$ is a blowup of $Y$ whose center is a
subvariety.
\end{df}

\begin{df}
Let $\pi : W \rightarrow V$ be a birational morphism with center the
scheme whose ideal sheaf is $\mathfrak{I}$, and let $D$ be an
irreducible divisor on $V$. We define the \emph{proper transformation of
$D$} to be
\[
  \overline{\pi^{-1} (D \cap U)},
\]
where $U= V \smallsetminus Z \left( \mathfrak{I}\right)$, and where $Z \left(
\mathfrak{I}\right)$ is the underlying subvariety that is the zero set
of the ideal~$\mathfrak{I}$.
\end{df}

We now assume that $V_0,\ldots,V_k$ is a sequence of varieties as in
Theorem~\ref{indeterminacy} which resolves the indeterminacy of
$\phi_0$. We let $V=V_k$, we write $\pi_V : V \rightarrow \pp^n$ for
the birational morphism that is a composition of monoidal
transformations, and we let $\phi = \phi_0 \circ \pi_V:V \rightarrow
\pp^n$ be the morphism extending~$\phi_0$.
We prove a lemma describing the Picard group of~$V$.

\begin{lem}
\label{indepofexceptdivs}
With notation as above, let $H_V\in\Div(V)$ be the proper transform
in~$V$ of~$H$, and for each $1\le i\le k$, let~$E_i\in\Div(V)$ be the
proper transform in~$V$ of the exceptional divisor of the monoidal
transformation~$V_i\to V_{i-1}$. Then
\[
  \Pic(V) = \zz H_V \oplus \zz E_1 \oplus \zz E_2
    \oplus \dots \oplus \zz E_k.
\]
In other words, $\Pic(V)$ is a free abelian group of rank $k+1$ generated
by~$H_V$ and~$E_1,\ldots,E_k$.
\end{lem}
\begin{proof}
Let $\pi:\tilde X\to X$ be the blowup of a smooth variety along a
smooth subvariety and let~$E$ denotes the exceptional divisor of the
blowup. Then it is well known that $\Pic(\tilde
X)=\pi^*\Pic(X)\oplus\zz E$; see for example~\cite[II.8.Ex~5]{H}.
Applying this fact repeatedly, starting from~$V_0=\pp^n$
and~$\Pic(V_0)=\zz H$, an easy induction gives the desired result.
\end{proof}

Similarly, for $f^{-1}$, we let $W=W_l$ be a variety which
resolves the indeterminacy of $\psi_0$, we let $\pi_W : W \rightarrow
\pp^n$ be a birational morphism that is a composition of monoidal
transformations, we write $\psi=\psi_0 \circ \pi_W:W \rightarrow
\pp^n$ for the morphism extending~$\psi$, and we let
\[
  \Pic(W) = \zz H_W \oplus \zz F_1 \oplus \zz F_2
    \oplus \dots \oplus \zz F_l,
\]
where $H_W\in\Div(W)$ is the proper transform
in~$W$ of~$H$, and for each $1\le j\le l$, let~$F_j\in\Div(W)$ be the
proper transform in~$W$ of the exceptional divisor of the monoidal
transformation~$W_j\to W_{j-1}$.

This is summarized in the following commutative diagrams.
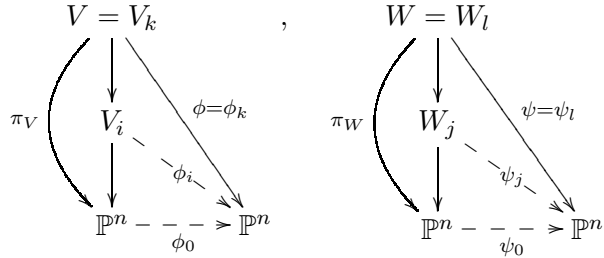
\[
   \xymatrix{
      V= V_k \ar[d] \ar[rdd]^{\phi=\phi_k} \ar@/_2pc/[dd]_{\pi_V} \\
      V_i \ar[d] \ar@{-->}[rd]|{\phi_i} \\
      \pp^n \ar@{-->}[r]_{\phi_0}  &\pp^n }
      , \quad
   \xymatrix{
      W=W_l \ar[d] \ar[rdd]^{\psi = \psi_l} \ar@/_2pc/[dd]_{\pi_W} \\
      W_j \ar[d] \ar@{-->}[rd]|{\psi_j} \\
      \pp^n \ar@{-->}[r]_{\psi_0}  &\pp^n }
\]

Here are some lemmas which will help us to define the essential
divisor and to prove Kawaguchi's conjecture in
section~\ref{proofkawconj}.

\begin{lem}\label{inclusion}
\[
  \phi_0 (H \smallsetminus Z(\phi_0)) \subset  Z(\psi_0).
\]
\end{lem}

\begin{proof}
Let
\[
  \phi_0=( X_0^d ,F_1 , \cdots , F_n ), \quad
  \psi_0 = ( X_0^e , G_1 , \cdots , G_n ).
\]
Then using the fact that~$\phi_0$ and~$\psi_0$ are inverse maps, we have
\[
  \psi_0\circ\phi_0
  = \left( X_0^{de},G_1(X_0^d,F_1,\cdots,F_n),\cdots \right)
  = \left( X_0^{de},X_0^{de-1}X_1 , \cdots, X_0^{de-1}X_n \right).
\]
Now let $P=[0,x_1,\cdots,x_n] \in H \smallsetminus Z(\phi_0)$. Then,
\[
  \psi_0\left( \phi_0(P) \right)
  = \left( 0,G_1(\phi_0(P)),\cdots \right)
  = \left( 0^{de},0^{de-1}X_1 , \cdots, 0^{de-1}X_n \right),
\]
and hence $\phi_0(P) \in Z(\psi_0)$.
\end{proof}

\begin{lem}\label{identity}
\[{\phi}_* {\phi}^*H = H.\]
\end{lem}
\begin{proof}
For this proof, we let $H$, $H_V$ and $E_i$ be specific closed
subvarieties of codimension 1, not linear equivalence classes.  Let
$C_i$ are the image of $E_i$ by $\phi$. Since $\phi$ is an
automorphism on
\[
  \af^n
  \cong V \smallsetminus \left( H_V \cup \left(\bigcup_{i=1}^k E_i\right) \right)
  \cong \pp^n \smallsetminus H,
\]
the $C_i$ are algebraic subsets of $H$. So we can choose a
hyperplane $H'$ which does not contain any of the $C_i$.  Let $H'^V$ be
the preimage of $H'$ by $\phi$. Then, $H'^V$ does not contain any of
$E_i$.  It also means that the $H'^V \cap E_i$ always has
codimension larger than $1$, since $E_i$ is irreducible. And the
codimension of $H'^V\cap H_V$ is also larger than 1. So,
\[
  \overline{\phi\left( H'^V \cap \left( H_V \cup \left(\bigcup_{i=1}^k
  E_i\right) \right) \right)}.
\]
is a closed cycle of codimension larger than $1$, since $\phi_*$ is
a graded group homomorphism. Therefore,
\begin{eqnarray*}
   \phi_*(H'^V ) &=& [K(H'^V):K(\phi(H'^V))] \cdot \overline{\phi(H'^V  )} \\
   &=&  [K(H'^V):K(\phi(H'^V))]
    \cdot \overline{\phi\left( H'^V \smallsetminus \left( H_V \cup
    \bigcup_{i=1}^k E_i \right) \right)}.
\end{eqnarray*}
Moreover, since $\phi$ is one-to-one outside of $H_V$ and $E_i$,
\[
  [K(H'^V):K(\phi(H'^V))]= 1
  \quad {\rm and} \quad
  \phi\left( H'^V \cap \left( H_V \cup
    \bigcup_{i=1}^k E_i \right) \right)
    = H' \smallsetminus H.
\]
Therefore, $\phi_*\phi^*H' = H'$, and hence $\phi_*\phi^*H =
\phi_*\phi^*H' = H'= H$.
\end{proof}

\begin{lem}\label{EinH}
\[
  \phi(H_V)\subset H \quad \text{and} \quad \phi(E_i) \subset H.
\]
\end{lem}
\begin{proof}
Suppose that there is a point $ Q \in H_V \cup \left( \bigcup_i E_i
\right) $ such that $\phi(Q) \in \af^n = \pp^n \smallsetminus H$.  Then, since
$Z(\psi_0) \subset H$, $\phi(Q) \not\in Z(\psi_0)$, and hence
$\psi_0(\phi(Q)) \in \af^n$.

We now consider some rational maps. Let $\widetilde{\phi} :V
\dashrightarrow W$ and $\widetilde{\psi} :W \dashrightarrow V$ be
rational maps that extend~$\phi$ and~$\psi$, and let $\zeta$ be the
composition of $\widetilde{\psi}$ and $\widetilde{\phi}$.
\[\xymatrixcolsep{5pc}
   \xymatrix{
   V \ar[d]_{\pi_V} \ar[rd]^{\phi} \ar@/^2pc/[rr]^{\zeta}  \ar@{-->}[r]^{\widetilde{\phi}}
   & W \ar[d]^{\pi_W} \ar[rd]^{\psi} \ar@{-->}[r]^{\widetilde{\psi}}
   & V \ar[d]^{\pi_V} \\
   \pp^n \ar@{-->}[r]_{\phi_0}  \ar@{-->}[d]
   & \pp^n \ar@{-->}[r]_{\psi_0}  \ar@{-->}[d]
   & \pp^n  \ar@{-->}[d] \\
   \af^n \ar@{-->}[r]_{f} \ar@/_2pc/[rr]_{f^{-1} \circ f=id_{\af^n}}
   & \af^n \ar@{-->}[r]_{f^{-1}}
   & \af^n
   }
\]
It is clear that $\psi_0 \circ \phi = \pi_V \circ \zeta$ where they
are defined. Further, $\zeta : V \dashrightarrow V$ is a rational map
which satisfies
\[
  \left( \pi_V \circ \zeta \circ \pi_V^{-1} \right)\big|_{\af^n} = f^{-1}
  \circ f = id_{\af^n}.
\]
But we have a trivial extended morphism $id_V : V \rightarrow V$
which also satisfies
\[
  \left( \pi_V \circ id_V \circ \pi_V^{-1} \right)_{\af^n} =  id_{\af^n},
\]
so that $\zeta = id_V$. Therefore,
\[
  \psi_0\circ\phi(Q) = \pi_V \circ \zeta(Q) = \pi_V(Q) \in H,
\]
which contradicts the assumption.
\end{proof}

\begin{lem}\label{b_j}
Write~$\phi^*H$ in terms of the basis for~$\Pic(V)$, say
\[
  \phi^*H = b_0H_V + \sum_{i=1}^k b_i E_i.
\]
Then $b_i>0$ for all $i$.
\end{lem}
\begin{proof}
Let $u$ be a uniformizer for the hyperplane $H$ at infinity,
\[
  u(x_0,\cdots, x_n)=x_0.
\]
Then by definition,
\[
  \phi^*H =\ord_{H_V}(u\circ\phi) \cdot H_V + \sum
  \ord_{E_i}(u\circ\phi) \cdot E_i .
\]
Furthermore, $u=0$ on $H$ because it is a uniformizer of $H$, while
$\phi(E_i)\subset H$ and $\phi(H_V) \subset H$ by Lemma~\ref{EinH}.
Therefore,
\[
  u \circ \phi =0
\]
on $H_V$ and on $E_i$, and hence $\ord_{H_V}(u\circ\phi) \geq 1$
and  $\ord_{E_i}(u\circ\phi) \geq 1$.
\end{proof}

\begin{thm}
\label{essential}
Write
\[
   \phi^*H = b_0H_V + \sum_{i=1}^k b_i E_i
\]
as in Lemma~$\ref{b_j}$.  Then there exists a unique index~$t \neq 0$
such that
\[
  b_t=1,\quad \phi_* E_t=H
  \quad\text{and}\quad \phi_*E_j=0\quad\text{for all $j\ne t$.}
\]
\end{thm}

\begin{proof} It is clear that ${\phi_*}H_V$ and
${\phi_*}E_i$ are nonnegative multiples of $H$, since $\Pic(\pp^n)
=\mathbb{Z}$. So we may write ${\phi}_*E_i=s_i H$ with $s_i
\geq 0$.  For any irreducible divisor~$D$,
the definition of ${\phi}_*D$ is
\[
   [K(D):K({\phi}(D))] \cdot \overline{ \phi(D)}.
\]
From this definition, we get
\begin{equation}
  \label{phiHV0}
  \phi_*H_V=0,
\end{equation}
because $\phi_0 (H\smallsetminus Z(\phi_0)) \subset Z(\psi_0)$ from
Lemma~\ref{inclusion}, and $Z(\psi_0)$ has codimension greater than $1$.
Furthermore, since ${\phi}_* {\phi}^*H = H$ from Lemma~\ref{identity},
we have
\[
  H = {\phi}_* {\phi}^*H
  = {\phi}_*\left(aH_V + \sum_{i=1}^k b_i E_i\right)
  = \sum_{i=1}^k  b_i {\phi}_* E_i = \sum_{i=1}^k b_is_i H.
\]
Lemma~\ref{b_j} says that $b_j>0$ for all $j$, and also $s_i\ge0$ for
all~$i$.  Therefore there is exactly one $t$ satisfying $s_t b_t=1$,
and $s_jb_j=0$ for all $j \neq t$. Then the fact that every $b_j>0$
implies that $s_j=0$ for all $j \neq t$. Finally, since~$s_t$
and~$b_t$ are non-negative integers, the equality~$s_tb_t=1$ implies
that $s_t=b_t=1$.
\end{proof}

\begin{df}
Let $E_t$ be the unique exceptional divisor satisfying
$\phi_*E_t=H$ as described in Theorem~$\ref{essential}$. We call $E_t$
the \emph{essential exceptional divisor} of $\phi$.
\end{df}

\section{Proof of Kawaguchi's Conjecture}
\label{proofkawconj}

\begin{thm}
\label{mainthm}
Let $f$ be a regular affine automorphism of $\af^n$. Then there is a
constant $C=C(f)$ such that for all $P\in\af^n$,
\[
   \dfrac{1}{\deg f} h(f(P)) + \dfrac{1}{\deg f^{-1}}h(f^{-1}(P))
   \ge \left( 1+ \dfrac{1}{\deg f \deg f^{-1}} \right) h(P) + C.
\]
\end{thm}

We recall an important definition.

\begin{df}
Let $D \in \Pic(X)$. We say that $D$ is \emph{numerically effective} if
\[
  D^r \cdot Y \geq 0
\]
for all $r\ge1$ and all integral subschemes $Y\subset$ of dimension~$r$.
\end{df}

\begin{prop}
Let $\zeta : X \rightarrow X'$ be a morphism of projective varieties,
and let $D \in \Pic(X')$ be numerically effective. Then $\zeta^*D$ is
also numerically effective.
\end{prop}
\begin{proof}
This is a standard result, see for example
\cite[Example~1.4.4(i)]{Laz}, but for the readers conveience, we
provide the short proof.  From the projection property of
intersection,
\begin{align*}
  (\zeta^*D)^r \cdot Y'
  =  \zeta^*D \cdot \bigl( (\zeta^*D)^{r-1} \cdot Y'\bigr)
  &= \zeta^* \left[ D \cdot \zeta_* \left\{ (\zeta^*D)^{r-1} \cdot
   Y'\right\} \right]\\
  &= \zeta^* \left[ D \cdot D \cdot \zeta_* \left\{ (\zeta^*D)^{r-2}
   \cdot Y'\right\} \right]
  = \cdots = D^r \cdot \zeta_* Y'.
\end{align*}
Since $\dim(\zeta_* Y') = r$ and $D$ is numerically effective, this
last intersection is nonnegative.
\end{proof}

An important ingredient in the proof of Theorem~\ref{mainthm} is a
result that describes the pullback of the pushforward of a numerically
effective divisor.

\begin{lem}
\label{pull-push}
Let $\rho: \widetilde{X} \rightarrow X$ be a birational morphism and
let $Z$ be an divisor on $\widetilde{X}$ that is both effective and
numerically effective. Then $\rho^*\rho_*Z \geq Z$.
\end{lem}

\begin{proof}
Let $Z$ be an effective and numerically effective divisor, and let
\[
  \Pic (V) = \langle D_1, \cdots, D_r \rangle
\]
be a generating set for $\Pic(V)$, where the $D_i$ are effective
divisor of $V$. Then we can find a generating set
\[
  \Pic (W) = \langle D_1^\#, \cdots, D_r^\#, F_1, \cdots, F_s \rangle,
\]
where $D_i^\#$ is the proper transformation of $D_i$, and where the
$F_i$ are exceptional divisors of the birational morphism~$\rho$.
We write the pull-back of~$D_i$ as
\[
  \rho^*D_i = D_i^\# + \sum_j m_{ij}F_j.
\]
This allows us to write the given divisor~$Z$ as
\begin{align*}
  Z &= \sum_i a_i D_i^\# + \sum_j b_j F_j \\
  &= \sum_i a_i\left(\rho^*D_i + \sum_j m_{ij}F_j\right) + \sum_j b_j F_j \\
  &= \sum_i a_i \rho^*D_i +  \sum_j f_j F_j,
\end{align*}
where each~$f_j$ is some expression involving the~$b_j$ and~$m_{ij}$.
Further, and this is a key point, the fact that~$Z$ is numerically
effective implies that $f_j \leq 0$ for all $j$. This follows
from~\cite[Lemma 2.19]{KD}.
\par
We now compute
\begin{align*}
  \rho^*\rho_*Z &= \sum_i a_i \rho^*\rho_*\rho^* D_i + \sum_j f_j
  \rho^*\rho_* F_j \\
  &= \sum_i a_i \rho^*D_i
  \qquad\text{since $\rho_*\rho^*D_i=D_i$ and $\rho_*F_j=0$.}
\end{align*}
Hence
\[
  \rho^*\rho_*Z - Z = -\sum_j f_j F_j = \sum_j (-f_j) F_j \ge 0
\]
is effective.
\end{proof}

We recall that we have the following diagram of maps,
\[
  \xymatrix{
  V \ar[d]_{\pi_{V}} \ar[rd]^{{\phi}} \\
  \pp^n \ar@{-->}[r]_{\phi_0}  & \pp^n }
\]
For notational convenience, we let
\[
  d=\deg(f)\qquad\text{and}\qquad d'=\deg(f^{-1}).
\]

\begin{lem}
\label{pullback}
Let~$E_t$ be the essential divisor of~$\phi$.
Then the pull-backs of~$H$ by~$\pi_V$ and~$\phi$ have the form
\[
  \pi_V^*H = H_V + d' E_t + \sum_{i\ne t} a_iE_i, \qquad
  \phi^*H = d H_V + E_t +  \sum_{i\ne t} b_iE_i,
\]
where the coefficients of the non-essential exceptional divisors satisfy
\[
  d'b_i \ge a_i \ge 0 \quad\text{for all $i\ne t$.}
\]
\end{lem}

\begin{proof}[Proof of Lemma $\ref{pullback}$]
Let
\[
  \pi_V^*H = a_0H_V + \sum_{i=1}^k a_i E_i \qquad\text{and}\qquad
   \phi^*H = b_0H_V +  \sum_{i=1}^k b_i E_i,
\]
Clearly ${\pi_V}_* \pi_V^*H = H$ by Lemma~\ref{identity}, and
${\pi_V}_*E_i=0$ because ${\pi_V}(E_i)\subset Z(\phi_0)$.  So,
\[
  H = {\pi_V}_*\pi_V^*H = {\pi_V}_*\left(a_0H_V+\sum_{i=1}^k a_iE_i\right) =
  {\pi_V}_*(a_0H_V)=a_0 H,
\]
and hence $a_0=1$. And we have $b_t=1$ by the definition of the
essential exceptional divisor.

Next we compute $b_0$ and $a_t$.  Letting $u$ be a uniformizer at $H$,
we have by definition
\[
  \phi^*H = \ord_{H_V}(u\circ \phi) \cdot H_V + \sum_{i=1}^k
  \ord_{E_i}(u\circ \phi) \cdot E_i.
\]
Taking the push-forward of $\phi^*H$ by $\pi_V$, we get
\[
  {\pi_V}_*\phi^*H = \ord_{H_V}(u \circ \phi) \cdot {\pi_V}_*H_V +
  \sum_{i=1}^k \ord_{E_i}(u\circ \phi) \cdot {\pi_V}_* E_i
  =\ord_{H_V}(u\circ \phi) \cdot H_V,
\]
since we showed
\[
  {\pi_V}_*H_V=H \qquad \text{and} \qquad {\pi_V}_*E_i=0.
\]
Furthermore, because $\phi = \phi_0\circ \pi_V $ and $\phi_0 =[x_0^d ,
f_1, \cdots, f_n]$, we have
\[
  \ord_{H_V}(u \circ \phi) = \ord_{H}(u \circ \phi \circ \pi_V^{-1}) =
  \ord_{H}(u \circ \phi_0)= d.
\]
On the other hand,
\[
  {\pi_V}_*\phi^*H = {\pi_V}_*\left(b_0H_V + \sum_{i=1}^k b_iE_i\right)
   = b_0  {\pi_V}_*(H_V) + \sum_{i=1}^k b_i{\pi_V}_*(E_i) = b_0H,
\]
and hence $b_0=d$.

Similarly,
\[
  \pi_V^*H = \ord_{H_V}(u\circ \pi_V) \cdot H_V + \sum_{i=1}^k
  \ord_{E_i}(u\circ \pi_V) \cdot E_i,
\]
and
\[
  {\phi}_*\pi_V^*H = \ord_{H_V}(u \circ \pi_V) \cdot {\phi}_*H_V +
  \sum_{i=1}^k \ord_{E_i}(u\circ \pi_V)\cdot {\phi}_* E_i
  =\ord_{E_t} (u\circ \phi)\cdot  H,
\]
since
\[
  {\phi}_*H_V=0, \qquad {\phi}_*E_t=H \qquad\text{and} \qquad \phi_*E_j =0
  \qquad\text{for all}\quad j\neq t.
\]
Furthermore, because
\[
  \pi_V = \psi_0 \circ \phi\qquad\ord_{E_t}(u\circ
  \phi)=b_V=1,\qquad\text{and}\qquad \phi_0 =[x_0^{d'} , g_1, \cdots, g_n],
\]
we have
\[
  \ord_{E_t}(u \circ \pi_V) = \ord_{E_t}(u \circ \psi_0 \circ \phi) =
  \ord_{H}(u \circ \psi_0)= d'.
\]
On the other hand,
\[
  {\phi}_*\pi_V^*H = {\phi}_*\left(a_0H_V + \sum_{i=1}^k a_iE_i\right) = a_0
  {\phi}_*H_V + \sum_{i=1}^k a_i {\phi}_*E_i = a_tH,
\]
and hence $a_t=d'$.
\par
It is clear that all of the $a_i$ and $b_i$ are non-negative, since~$H$
is effective and $H_V$ and the~$E_i$ are the divisors whose support
is contained in~$\pi_V^{-1}(H)$ and~$\phi^{-1}(H)$. It remains to prove
the inequality~$d'b_i\ge a_i$. We apply Lemma~\ref{pull-push}
to the divisor~$\pi_V^*H$, which is both effective and numerically
effective, and to the map~$\phi_V$. Lemma~\ref{pull-push} tells us that
\begin{align*}
   \phi_V^*{\phi_V}_*\bigl(\pi_V^*(H)\bigr) \geq\pi_V^*(H).
\end{align*}
We also have
\begin{align*}
  {\phi_V}_*\pi_V^*(H)
  &= {\phi_V}_*\left(H_V+d'E_t+\sum_{i\ne t}a_iE_i\right)
    &&\text{from above,}\\
  &= d'H
    &&\text{from Theorem \ref{essential}.}
\end{align*}
(Note that $\phi_*E_i=0$ for $i\ne t$ be the definition of the
essential divisor, and we have already seen in~\eqref{phiHV0} that
${\phi_V}_*H_V=0$.)
Hence
\[
  \phi_V^*(d'H) = \phi_V^*({\phi_V}_*\pi_V^*(H)) \geq\pi_V^*(H).
\]
Using the expressions for $\phi_V^*(H)$ and $\pi_V^*(H)$ from
above, this inequality gives
\[
  d'\left(d H_V + E_t + \sum_{i\ne t} a_iE_i\right)
  \ge H_V + d' E_t +  \sum_{i\ne t} b_iE_i.
\]
A little bit of algebra yields
\[
  (dd'-1)H_V + \sum_{i\ne t} (d'a_i-b_i)E_i \ge 0.
\]
\par
Then the following lemma shows that $d'a_i\ge b_i$ for all $i\ne t$.
\end{proof}

\begin{lem}
The divisor
\[
   c_0H_V + \sum_{i\ne t} c_iE_i \in \Div(V)
\]
is linearly equivalent to an effective divisor if and only if
$c_i\ge0$ for all~$i\neq t$.
\end{lem}
\begin{proof}
One direction is obvious. For the other, suppose that
\[
  c_0H_V + \sum_{i\ne t} c_iE_i
  \sim n_0H_V + \sum_i n_iE_i + \sum_j m_jD_j,
\]
where $n_i\ge0$ and $m_j\ge0$, and where the~$D_j$ are irreducible
divisors distinct from~$H_V$ and the~$E_i$. Note that the~$D_j$ have
the property that $\pi_*D_j$ has nontrivial intersection with~$\af^n$,
since~$H_V$ is the proper transform of $H=\pp^n\setminus\af^n$ and
the~$E_i$ are the exceptional divisors of the blowup
$\pi:V\to\pp^n$. Then the fact that~$\phi_0$ is an automorphsim
of~$\af^n$ implies that $\phi_*D_j\ne0$.  Hence there are positive
integers $k_j$ such that $\phi_*D_j\sim k_jH$.
\par
We know from before that $\phi_*H_V=0$ and $\phi_*E_i=0$ for $i\ne t$,
and also $\phi_*E_t=H$. Therefore
\[
  0 = \phi_*\biggl(c_0H_V + \sum_{i\ne t} c_iE_i\biggr)
  = n_t\phi_*E_t +  \sum_j m_j\phi_*D_j
  \sim n_t H +  \sum_j m_jk_j H.
\]
Here $n_t\ge0$, $m_j\ge0$, and $k_j>0$, and $\Pic(\pp^n)=\zz$,
so we conclude that $n_t=0$ and $m_j=0$ for all~$j$.
\par
This proves that
\[
  c_0H_V + \sum_{i\ne t} c_iE_i
  \sim n_0H_V + \sum_{i\ne t} n_iE_i.
\]
Lemma~\ref{indepofexceptdivs} says that $H_V$ and the $E_i$ are
linearly independent in~$\Pic(V)$, so $c_i=n_i\ge0$ for all $i$.
\end{proof}

\begin{proof}[Proof of Theorem~$\ref{mainthm}$]
To ease notation, we let~$E_V=E_t$ be the essential divisor,
and we define
\[
  M_V = \sum_{i\neq t} a_iE_i\qquad\text{and} \qquad
  I_V = \sum_{i\neq t} b_i E_i.
\]
Then Lemma~\ref{pull-push} can be rewritten as
\[
  \pi_V^*H = H_V + d' E_V + M_V,\qquad
  \phi^*H = d H_V + E_V + I_V,
\]
where $M_V,I_V\ge0$ and $d'I_V\ge M_V$. Further, the supports of~$I_V$
and~$M_V$ do not contain the support of~$H_V$ or~$E_V$.

Now do the same with a blowup that resolves the indeterminacy of $\psi_0$,
\[
  \xymatrix{
	 & W \ar[d]^{\pi_{W}} \ar[dl]_{{\psi}} \\
  \pp^n   &  \pp^n \ar@{-->}[l]^{ \psi_0} }
\]
Since $\psi_0$ is the extension of the affine automorphism
$f^{-1}$ to~$\pp^n$, we have using similar notation,
\[
  \pi_W^*H = H_W + d'F_W + N_W,\qquad
  \psi^*H = dH_W + E_W + J_W,
\]
with $J_W,N_W\ge0$ and $dJ_W \geq N_W$.  Here $F_W$ is the essential
divisor of $\psi$, and $N_W$ and $J_W$ are $\af^n$-effective divisors
whose supports do not contain the support of $H_W$ or $F_W$.

Now, consider a blowup $U$ of $\pp^n$ that resolves both $\phi_0$ and
$\psi_0$. The assumption that~$f$ is regular implies that the centers
of the blowups~$V$ and~$W$ are disjoint, so~$U$ is a blowup of $\pp^n$
whose center is the scheme-theoretic sum of the centers of the blowups
$V$ and $W$. We have the following diagram:
\[
\xymatrix{
   & & U \ar[rd] \ar[ld] \ar[dd]^{\pi_U} \ar@/_2pc/[lldd]_{\psi_U} \ar@/^2pc/[rrdd]^{\phi_U}\\
   &W \ar[rd]^<<{\pi_W} \ar[ld]_<<<{\psi}& &
   V \ar[rd]^<<{\phi} \ar[ld]_<<<{\pi_V} \\
\pp^n & & \pp^n \ar@{-->}[ll]^{\psi_0} \ar@{-->}[rr]_{\phi_0}& & \pp^n
}
\]
Let $H_U$ be the proper transformation of $H$ by $\pi_U : U
\rightarrow \pp^n$, and let $E,M,I, F,N,J$ be the proper
transformations of $E_V,M_V,I_V, F_W,N_W,J_W$, respectively. Then we
have
\begin{align*}
  \pi_U^*H  &=  H_U + d'E + dF + M + N,\\
  \phi_U^*H &=  d(H_U + dF + N)+ E + I,\\
  \psi_U^*H &=  d'(H_U  + d'E + M ) +F + J.
\end{align*}

Finally, we compute the divisor
\begin{align*}
  D &= \dfrac{1}{d} \phi_U^*H + \dfrac{1}{d'} \psi_U^*H - \left( 1 +
  \dfrac{1}{dd'} \right) \pi_U^*H \\
  &= (H_U + dF + N) + \dfrac{1}{d} (E+I) + (H_U + d'E + M ) +
  \dfrac{1}{d'}(F+J)\\
  &\qquad {}- \left( 1 + \dfrac{1}{dd'}\right) (H_U + d'E + dF + M + N) \\
  &= \left(1-\dfrac{1}{dd'} \right)H_U
      + \dfrac{1}{dd'}(dJ - N + d'I - M).
\end{align*}
The fact that $d'I>M$ and $dJ >N$ implies that $D$ is effective,
which completes the proof of Theorem~\ref{mainthm}.
\end{proof}

\section{Canonical Height Functions for Regular Affine Automorphism}

In Sections~4 and 5 of \cite{K}, Kawaguchi constructed canonical
heights for regular affine automorphisms under the assumption that his
height lower bound conjecture was true.  Since we have proved his
conjecture, his construction of canonical heights, as described in the
following theorem, is now valid in all dimensions.

\begin{thm}
Let $f :\af^n \rightarrow \af^n$ be an affine
automorphism of degree $d$ and $f^{-1}$ be its inverse map
of degree $d'$. Define
\[
  \widehat{h}_+(P) = \limsup \dfrac{1}{d^l}h(f^n(P)), \quad
  \widehat{h}_-(P) = \limsup \dfrac{1}{d'^l}(h(f^{-1})^n(P))
\]
and
\[
  \widehat{h}(P) = \widehat{h}_+(P) - \widehat{h}_-(P)
\]
Then,
\[
  \dfrac{1}{d}\widehat{h}(f(P))+ \dfrac{1}{d'}\widehat{h}(f^{-1}(P)) =
  \left( 1+ \dfrac{1}{dd'}\right) \widehat{h}(P)
\]
and $\widehat{h}(P)=0$ if and only if $P$ is $f$-periodic.
\end{thm}

\section{An Example}

An example of a H\'{e}non map on~$\af^n$, which is the regular affine
automorphism, is
\[
  f(x_1,\cdots,x_n) = (x_2, x_3+x_2^2, x_4+x_3^2, \cdots, x_n +
  x_{n-1}^2, x_1 + x_n^2).
\]
On $\af^2$, properties of H\'{e}non maps are well known, but less is
known for H\'{e}non maps on $\af^n$.  In this section we describe what
happens for a H\'{e}non map of degree 2 on $\af^3$.  Since the details
of the proof are somewhat lengthy, we just give a sketch and refer the
reader to~\cite{Lee} for details.

Let $f$ be a H\'{e}non map defined as
\[
  (x,y,z) = (y,z+y^2, x+z^2).
\]
It's an affine automorphism with inverse
\[
  f^{-1} = (z-(y-x^2), x, y-x^2)
\]
and meromorphic extensions
\begin{align*}
  \phi_0([x,y,z,w]) &= [yw,zw+y^2,xw+z^2,w^2],\\
  \psi_0([x,y,z,w]) &= [zw^3-(yw-x^2), xw^3, yw^3-x^2w^2,w^4].
\end{align*}
The map~$f$ is regular, since the indeterminacy loci are
$Z(\phi_0)=\{[1,0,0,0]\}$ and $Z(\psi_0)=\{[0,y,z,0]\}$.

After a number of blow-ups along subvarieties of dimension zero and one,
and using intersection theory to compute relevant coefficients, we
find that the resolution of~$f$ gives a map $\phi_V:V\to\pp^3$ with
\begin{align*}
  \pi_V^*H  &= H + E_1 + 2E_2 + 2E_3 + 4E_4 + 4E_5, \\
  \phi_V^*H &=  2H + E_1 + 2E_2 + E_3 + 2E_4 + E_5.
\end{align*}
Further,~$E_5$ is the essential exceptional divisor.
Similarly, resolving~$f^{-1}$ gives a map $\psi_W:W\to\pp^3$ with
\begin{align*}
  \pi_W^*(H) &= H + F_1 + 2F_3 + 2F_3 + 2F_4 + F_5 + 2F_6+ 2F_7, \\
  \psi_W^*(H) &= 4H + 2F_1 + 4F_2 + 2F_3 + F_4 + F_5 + 2F_6 + F_7.
\end{align*}
Combining the two resolutions to form a single variety that resolves
both~$f$ and~$f^{-1}$ yields
\begin{align*}
  \pi^*H &= H + 4E_5 + 2F_7 + (E_1 + 2E_2 + 2E_3 + 4E_4) + (F_1 +
      2F_2 + 2F_3 + F_5 + 2F_6 + 2F_7), \\
  \phi^*H &= 2(H+2F_7+(F_1 + 2F_2 + 2F_3 + F_5 + 2F_6 + 2F_7)) + E_5
      + (E_1 + 2E_2 + E_3 + 2E_4 ), \\
   \psi^*H &= 4(H+4E_5 + (E_1 + 2E_2 + 2E_3 + 4E_4)) + F_4 + (2F_1 +
       4F_2 + 2F_3 + F_5 + 2F_6 + F_7 ),
\end{align*}
and hence
\begin{align*}
   \frac{1}{2}\phi^*&H + \frac{1}{4}\psi^*H - \left(1+\dfrac{1}{8}\right)\pi^*H \\
   & = \dfrac{7}{8} H + \dfrac{3}{8}E_1 + \dfrac{3}{4}E_2 +
       \dfrac{1}{4}E_3 + \dfrac{1}{2}E_4 + 0E_5
       + \dfrac{3}{8}F_1 + \dfrac{3}{4}F_2 + \dfrac{1}{4}F_3 + 0F_4 +
       \dfrac{1}{8}F_5 + \dfrac{1}{4}F_6 + 0F_7 \\
  &\ge 0.
\end{align*}


\end{document}